\documentclass[11pt,a4paper]{article}
\usepackage{graphicx}
\usepackage{amsmath,amsfonts,amsthm}
\usepackage{xfrac,esint,mathrsfs,color}
\usepackage[utf8]{inputenc}
\usepackage[colorlinks=true]{hyperref}
\usepackage{esint}
\usepackage[top=4.5cm,bottom=4cm,left=3.3cm,right=3cm]{geometry}

\newcommand{\parametermu}{\mu}

\newcommand{\C}{\mathbb{C}}

\newcommand\noop[1]{}
\definecolor{green}{rgb}{0,.5,0}
\newcommand\eps{\varepsilon}

\newcommand\Div{\mathrm{Div\,}}

\newcommand\Id{\mathrm{Id}}
\newcommand\sym{\mathrm{sym}}
\newcommand\loc{\mathrm{loc}}

\newcommand\Lip{\mathrm{Lip\,}}

\newcommand\R{\mathbb{R}}
\newcommand\N{\mathbb{N}}
\newcommand\Z{\mathbb{Z}}
\newcommand\calH{\mathcal{H}}

\newcommand\calA{\mathcal{A}}

\newcommand\calB{\mathcal{B}}
\newcommand\calL{\mathcal{L}}

\newcommand{\be}[1]{\begin{equation} #1 \end{equation}}

\newcommand{\fz}{f_0}

\newcommand{\Hn}{{\mathcal H}^{n-1}}

\newcommand{\cost}{\kappa}

\def\p{\varphi}

\def\R{\mathbb{R}}

\def\Z{\mathbb{Z}}
\def\N{\mathbb{N}}
\def\H{\mathcal{H}}

\def\Div{\textup{div}\,}

\def\wto{\rightharpoonup}
\def\weakto{\rightharpoonup}

\newtheorem{theorem}{Theorem}
\newtheorem{lemma}{Lemma}[section]
\newtheorem{corollary}[lemma]{Corollary}
\newtheorem{proposition}[lemma]{Proposition}
\newtheorem{remark}[lemma]{Remark}

\def\Carre#1#2{\vbox{
		\hrule height .#2pt
		\hbox{\vrule width .#2pt height #1pt \kern #1pt
			\vrule width .#2pt}
		\hrule height .#2pt}}

\begin{document}
\begin{center}
  {\Large Approximation of functions with small jump sets and existence of strong minimizers of Griffith's energy}
\\[5mm]
{\today}\\[5mm]
Antonin Chambolle$^{1}$, Sergio Conti$^{2}$, and Flaviana Iurlano$^{3}$\\[2mm]
{\em $^{1}$ CMAP, \'Ecole Polytechnique, CNRS,\\ 91128 Palaiseau Cedex, France}\\[1mm]
{\em $^{2}$ Institut f\"ur Angewandte Mathematik,
Universit\"at Bonn\\ 53115 Bonn, Germany}\\[1mm]
{\em $^{3}$ Laboratoire Jacques-Louis Lions, Universit\'e Paris 6\\ 75005 Paris, France}\\[3mm]
    \begin{minipage}[c]{0.8\textwidth}
We prove that special functions of bounded deformation with small jump set are close in energy to functions which are smooth in a slightly smaller domain. This permits to generalize the decay estimate by De Giorgi, Carriero, and Leaci to the linearized context in dimension $n$ and to establish the closedness of the jump set for {local} minimizers of the Griffith energy.
    \end{minipage}
\end{center}

\section{Introduction}
The last few years have seen the development of several techniques to approximate a certain class of special functions with bounded deformation ($SBD$). Such functions appear in the mathematical formulation of fracture in the framework of linearized elasticity~\cite{fra-mar,gbd}. Their peculiarity is the structure of the symmetric distributional derivative, which unveils the presence of a regularity zone, where the function admits a symmetric gradient in an approximate sense, and of a singularity zone, where the function jumps. 
The variational problems in such spaces are generally hard to tackle, and
few strong existence results are known with no artificial additional constraint
($L^\infty$ bounds, a priori bounds on the discontinuity set, ...)
We address in this paper the issue of the existence of strong minimizers for
Griffith's model of  brittle fracture in linear elasticity~\cite{Griffith1921}, in the formulation introduced by
Francfort and Marigo in~\cite{fra-mar} (up to lower-order terms, see Theorem~\ref{t:main} below for details).
A scalar simplification of this problem leads to the  Mumford-Shah functional~\cite{mum-shah}, well known in the mathematical literature.
By strong minimizers, we mean
functions defined in an open set, with a closed $(n-1)$-dimensional
discontinuity set.
This class of ``free-discontinuity problems'' has been thoroughly studied, in
the scalar case, in
the 90's and is now well understood; a classical way to the existence of strong
minimizers is to show first the existence of minimizers in the class
$SBV$~\cite{DeGiorgiAmbrosio}
and then, that the jump set of such solutions is
closed~\cite{DegiorgiCarrieroLeaci1989}. The technical point where
the proof of~\cite{DegiorgiCarrieroLeaci1989} (and most subsequent proofs,
see for instance~\cite{MaddalenaSolimini,BucurLuckhaus}) is not easily
transferred is an approximation issue where one needs to show that a
minimizer with almost no jump is close to a smooth minimizer. Hence
the need to develop new approximation methods in such a context.

Indeed, standard methods do not work in a linearized framework, being these based on the identification of bad parts of the function \emph{via} coarea formula and their removal \emph{via} truncation~\cite{DegiorgiCarrieroLeaci1989}. In contrast, results of this kind have important applications,
beyond the already mentioned existence of minimizers for the Griffith's problem in its strong formulation, 
including the integral representation in $SBD$ of functionals with $p$ growth,
or the study of the quasi-static evolution of brittle fracture, {in the $2$d case see respectively \cite{ContiFocardiIurlanoRepr}, \cite{ContiFocardiIurlano_exist}, and \cite{FriedrichSolombrino}.}

The first two authors, together with G.~Francfort
established in \cite{ChambolleContiFrancfort2016}
 a Poincaré-Korn's inequality for functions with $p$-integrable strain and small jump set. Their idea is to estimate the symmetric variation of $u$ on many lines having different orientations and not intersecting the jump set, thus allowing to use the fundamental theorem of calculus along such lines.
 Through a variant of this strategy with restrictions to the planar setting,
the last {two} authors, together with M.~Focardi,
 proved in \cite{ContiFocardiIurlanoRepr} that such functions are in fact Sobolev out of an exceptional set with small area and perimeter, and obtained a Korn's inequality in that class. 
Similar slicing techniques were first used in \cite{Chambolle2004a,Chambolle2004b,iur} to prove density and in \cite{ContiSchweizer2,DolzmannMueller,Kohn} to prove rigidity. 
A different approach, based on the idea of binding the jump heights after suitable
modifications of the jump set and of the displacement field, has been employed by Friedrich in \cite{Friedrich1,Friedrich2,Friedrich3} to prove Poincaré-Korn's, Korn's with a non-sharp exponent, and piecewise Korn's inequalities in the planar case.

On the one hand, the drawback of \cite{ChambolleContiFrancfort2016} is the lack of control on the perimeter of the exceptional set, which prevents good estimates for the strain; nevertheless these can be recovered through a suitable mollification. 
On the other hand, \cite{ContiFocardiIurlanoRepr} and \cite{Friedrich2,Friedrich3} use respectively a scaling argument and geometric constructions which hold in dimension $2$ and do not trivially extend to higher dimensions.

The purpose of this paper is to establish a $n$-dimensional version of the approximation result in \cite{ContiFocardiIurlanoRepr}, {where it is shown that $SBD$ functions with a small jump set can be approximated with $W^{1,p}$ functions. 
Our technique, which is slightly different from \cite{ContiFocardiIurlanoRepr} and other ``classical'' methods already developed for the approximation of ($G$)$SBD$ functions \cite{Chambolle2004a,Chambolle2004b,iur12,ContiFocardiIurlano_densp}, is based on a subdivision of the domain into ``bad'' and ``good'' little cubes (depending on the size of the jump set in each cube) which was first introduced in \cite{ChambolleContiFrancfort2017} and also recently used in \cite{ChambolleCrismale}.} 
Given a function with $p$-integrable strain and small jump set, we first cover the domain by dyadic small cubes which become smaller and smaller close to the boundary, we then identify the good cubes, those which still contain a small amount of jump. The biggest cubes are chosen with size much larger than the measure of total jump, so that they are all good, and give rise to a compact set which covers most of the domain.
In the good cubes the result in \cite{ChambolleContiFrancfort2016} provides a small set and an affine function which is close to the original function out of the exceptional set. In such set we redefine the function as the aforementioned affine function, then we mollify the new function in the good cube in order to gain regularity. Finally our approximations are obtained by taking a partition of unity on the good part and keeping the original function in the rest. Since a large compact is made of good cubes, our approximation is smooth in most of the domain. For more details we refer to Section \ref{s:app}. {Since there are no additional} mathematical difficulties, we prove the approximation in the more general setting $GSBD$, see Section \ref{s:not} for the definition. 

As we mentioned, this result can be
employed to prove the existence of strong minimizers for Griffith's energy, thus extending 
for $p=2$\footnote{{This only holds for $p=2$,} since \cite{ContiFocardiIurlano_exist} relies,
when $p\neq 2$, on integral estimates proved in~\cite{ContiFocardiIurlanoRegularity} which do not scale with the {appropriate} exponent in dimensions larger than 2.}
the result in \cite{ContiFocardiIurlano_exist} in any dimension. 
{Denoting
by $\C$ the ``Hooke's law'' of a linear-elastic material,
that is, $\C$ is a symmetric linear map from $\R^{n\times n}$ to itself
with the properties
\begin{equation}\label{eqassC}
\C(\xi-\xi^T)=0 \text{ and } \C \xi\cdot \xi \ge c_0 |\xi+\xi^T|^2 \text{ for all } \xi\in\R^{n\times n},
\end{equation}
{we obtain} the following result:}
\begin{theorem}\label{t:main}
	Let $\Omega\subset\R^n$ be a bounded Lipschitz set, $g\in L^\infty(\Omega{;\R^n})$, {$\beta>0$, $\cost>0$, $\C$ as in (\ref{eqassC}). }
	Then the functional
	\begin{equation}
	{{E_2}[\Gamma,u]:=\int_{\Omega\setminus\Gamma} (\C e(u):e(u)+\cost|u-g|^2) dx 
		+2\beta \calH^{n-1}(\Gamma\cap\Omega)}
	\end{equation}
	has a minimizer in the class
	\begin{equation}
	{\calA_2}:=\{(u,\Gamma): \Gamma\subset\overline \Omega \text{ closed, } u\in C^1(\Omega\setminus\Gamma)\}.
	\end{equation}
\end{theorem}
\noindent{Here, $e(u)= (Du+Du^T)/2$ is the symmetrized gradient of the
displacement $u$.}

{In \cite{ContiFocardiIurlano_exist}} the only {part of the argument which is restricted to two dimensions} was in fact the convergence of quasi-minimizers with vanishing jump to a minimizer without jump, which is obtained precisely using the $2$d approximation of \cite{ContiFocardiIurlanoRepr}. We show the convergence in Section \ref{s:lim}, making use of the approximation result in Section \ref{s:app}.
We refer to \cite{ContiFocardiIurlano_exist} for the rest of the proof of existence since it remains unchanged {in higher dimension}.

We remark that apparently the formulation of Griffith's problem considered in \cite{ContiFocardiIurlano_exist} and here differs from the original one for the presence of a fidelity term of type $|u-g|^2$, and for the absence of boundary conditions. Both of them have the only role of guaranteeing existence {of a minimizer in $GSBD^p$} for the weak global problem, hence one should employ different compactness and semicontinuity theorems in the two cases. {One uses \cite[Theorem 11.3]{gbd} in the first case; in the second case, in presence of Dirichlet boundary conditions, one uses \cite[Theorem 4.15]{FriedrichSolombrino} in dimension 2. Compactness in the $n$-dimensional $GSBD$ setting in presence of Dirichlet boundary conditions is still an open problem.} {If weak compactness holds, or, in other words, if the weak problem has a minimizer, then the present argument yields the desired regularity. This regularity result also holds in the case $\cost=0$, which coincides with the classical Griffith model.
 \begin{theorem}\label{t:mainreg}
	Let $\Omega\subset\R^n$ be a bounded Lipschitz set, $g\in L^\infty(\Omega;\R^n)$, $\beta>0$, $\cost\ge 0$, $\C$ as in (\ref{eqassC}).
	Let $u\in GSBD^2(\Omega)$ be a local minimizer of 
	\begin{equation}\label{eqgriffintrop}
	\int_{\Omega\setminus J_u} (\C e(u):e(u)+\cost|u-g|^2) dx 
		+2\beta \calH^{n-1}(J_u)
	\end{equation}
	Then, up to null sets, $u$ and $J_u$ coincide with a local minimizer of 
	$E_2[\Gamma,u]$ in the class
	$\calA_2$.
\end{theorem}
By local minimizer we mean here minimizer with respect to perturbations with compact support, i.e., in the class of all $v\in GSBD^2(\Omega)$ such that
$\{u\ne v\}\subset\subset\Omega$.}
  
The structure of the paper is the following. In Section \ref{s:not} we introduce the notation for $GSBD$ functions. In Section \ref{s:app} we state and prove our main result, the approximation in any dimension for functions with $p$-integrable strain and small jump set. In Section \ref{s:lim} we study the limit of quasi-minimizers with vanishing jump sets, which is instrumental to obtain existence of minimizers for Griffith's problem in any dimension. {Finally, in Section \ref{s:exis} we recall from \cite{ContiFocardiIurlano_exist} the main steps of the proof of Theorem \ref{t:main}.} 

\section{{Notation}}\label{s:not}
Fixed $\Omega\subset\R^n$ open and $u\in L^1(\Omega;\R^n)$ one defines the slice $u^\xi_y:\Omega^\xi_y\to\R$ for $\xi\in {\mathbb{S}}^{n-1}$ 
and $y\in\R^n$ by
$u^\xi_y(t)=u(y+t\xi)\cdot\xi$, where $\Omega^\xi_y:=\{t\in\R: y+t\xi\in\Omega\}$.
One also introduces $\Omega^\xi:=(\Id-\xi\otimes \xi)\Omega$, that is the orthogonal projection of $\Omega$ in the direction $\xi$. 

A generalized special function with bounded deformation $GSBD(\Omega)$ {(see \cite{gbd})} is then a $\mathcal{L}^n$-measurable function $u\colon\Omega\to\mathbb{R}^n$ for which there exists a bounded positive Radon measure $\lambda_u\in\mathcal{M}_b^+(\Omega)$
such that the following condition holds for every $\xi\in\mathbb{S}^{n-1}$: 
for $\H^{n-1}$-a.e.\ $y\in\Omega^\xi$ the function {$u^\xi_y(t)$
belongs to $SBV_\loc(\Omega^\xi_y)$} and for every Borel set $B\subset \Omega$ it satisfies
\begin{equation}\label{e:sl}
\int_{{\Omega}^\xi}\Big(|Du^\xi_y|(B^\xi_y\setminus J^1_{u^\xi_y})+\H^0(B^\xi_y\cap J^1_{u^\xi_y})\Big)d\H^{n-1}\leq
\lambda_u(B),
\end{equation}
where $J^1_{u^\xi_y}:=\{t\in J_{u^\xi_y}:|[u^\xi_y](t)|\geq1\}$.

If $u\in GSBD(\Omega)$, the approximate symmetric gradient $e(u)$ and the approximate jump set $J_u$ are well-defined, are respectively integrable and rectifiable, and can be reconstructed by slicing, see \cite{gbd} for details. We refer to \cite{AmbrosioFuscoPallara} for properties of functions of bounded variation $BV$. 

The subspace $GSBD^p(\Omega)$ {contains} all functions in $GSBD(\Omega)$ satisfying $e(u)\in L^p(\Omega;\R^{n{\times} n})$ and $\calH^{n-1}(J_u)<\infty$.
\section{Approximation {of functions with small jump}}\label{s:app}
\subsection{Preliminary results}

{We begin by stating a slight generalization of the Poincar\'e-Korn inequality for functions with small jump set obtained in \cite[Prop.~3]{ChambolleContiFrancfort2016}. }
\begin{proposition}\label{th:prop3b} 
Let $0<\theta''<\theta'<1$, $r>0$. 
Let $Q=(-r,r)^n$,
	$Q'=(-\theta' r,\theta' r)^n$, $p\in [1,\infty)$, $u\in GSBD^p(Q)$.
	\begin{enumerate}
		\item \label{propsobol1}
		There exists a set $\omega\subset Q'$ and an affine
		function $a:\R^n\to\R^n$ with $e(a)=0$ such that
		\begin{equation}\label{eq:stimom}
		|\omega|\le c_* r \H^{n-1}(J_u)
		\end{equation}
		and
		\begin{equation}\label{eq:stimu}
		\int_{Q'\setminus\omega} |u-a|^{np/(n-1)}
		\le c_*r^{n(p-1)/(n-1)}\left(\int_Q |e(u)|^p dx\right)^{n/(n-1)}
		\end{equation}
		\item \label{propsobol2}
		If additionally $p>1$ then there is $\bar p>0$ (depending on $p$ and $n$) such that,
		for a given mollifier $\rho_r\in C^\infty_c(B_{(\theta'-\theta'')r})$,  $\rho_r(x)=r^{-n}\rho_1(x/r)$,
		the function $v=u\chi_{Q'\setminus\omega}
		+a\chi_\omega$ obeys
		\begin{equation*}
		\int_{Q''} | e(v\ast \rho_r)- e(u)\ast\rho_r|^p dx \le c 
		\left(\frac{\H^{n-1}(J_u)}{r^{n-1}}\right)^{\bar p} \int_Q |e(u)|^p dx\,,
		\end{equation*}
		where $Q''=(-\theta'' r,\theta'' r)^n$.
	\end{enumerate}
	The constant in \ref{propsobol1}.~depends only on $p$, $n$ and $\theta'$, the one in
	\ref{propsobol2}.~also on $\rho_1$ and $\theta''$.
\end{proposition}
\begin{proof}
 This result was proven in  Prop.~3 of~\cite{ChambolleContiFrancfort2016} for $\theta'=1/2$, $\theta''=1/4$ and $u\in SBD^p$. 
 
The same proof works in  $u\in GSBD^p(Q)$, since it uses only the properties of slices, which hold also in the generalazed context (see \cite[Theorem 8.1 and 9.1]{gbd} for the slicing formulas for $J_u$ and $e(u)$ in $GBD$).  
At the same time, the proof works for general values of $\theta'$ and $\theta''$ after very minor corrections. Basically, {using the same notation of \cite{ChambolleContiFrancfort2016}}, the explicit bound $\H^1(R_\xi^x)\ge 1$ becomes
 $\H^1(R_\xi^x)\ge 2(1-\theta')$, correspondingly in the definition of $\omega_\xi^*$ one takes $(1-\theta')$ instead of $1/2$. The explicit constant in (3.5) becomes $1/(1-\theta')$, which then propagates (via the unspecified constants ``$c$'', which hereby acquire a dependence on $\theta'$) to the end of the proof. The proof of 2. is unchanged, since it already depends on the choice of the mollifier.
\end{proof}
\begin{remark}
 The statement still holds for $\theta'=1$, but we do not need this here.
\end{remark}

\subsection{{The main result}}
We first show how a $GSBD^p$ function  with (very) small jump set can be well
approximated by a function which is smooth on {a large subset} of the domain. For simplicity we assume that the domain is a cube $Q$, so that the coverings are all explicit, using a Whitney-type argument our proof extends easily to other regular open sets.  The key idea is to cover the domain $Q$ into a large number of small cubes of side length $\delta$, and then to let the decomposition refine towards the boundary. 
Then one applies the Korn-Poincar\'e estimate of Proposition \ref{th:prop3b} to each cube. 
If $\H^{n-1}(J_u)$ is sufficiently small, on a scale set by the length scale $\delta$ (and the constants of Proposition \ref{th:prop3b}), then for each of the  ``interior'' cubes (with side length $\delta$) one finds that the 
exceptional set covers only a small part of the cube. One replaces $u$ by the affine approximant in the exceptional set, mollifies, and then interpolates between neighbouring cubes, to obtain a smooth function with the needed properties. 
A bound on the difference between the affine approximants in neighbouring cubes is obtained by using Proposition \ref{th:prop3b}  on slightly enlarged versions of the cubes.

The boundary layer is however different. Moving towards the exterior boundary, one uses smaller and smaller cubes, and eventually it may happen that the exceptional set covers some cubes entirely. 
Then the rigidity estimate gives no information on the structure of $u$, and our construction cannot be performed.
The union of those cubes, denoted by $\mathcal B$ in the proof, is therefore treated differently:  the function $u$ is left untouched here. This generates a difficulty with the partition-of-unity approach, at the boundary between the two regions. The key idea here is to construct a partition of unity only in $Q\setminus\mathcal B$, with summands that do not have to obey any boundary data on {$\partial^* \mathcal B$} (but vanish as usual on $\partial Q$). Then one can completely separate the construction in $Q\setminus\mathcal B$, which is done by filling the holes and mollifying, from  the one on $\mathcal B$, where $u$ is untouched.
{As a result, we create a small amount of jump, which however remains
in a $\delta$-neighborhood of the original jump set and whose measure
remains controlled. {In 2D, it is possible to perform a similar construction
without any additional jump, see~\cite{ContiFocardiIurlanoRepr}.}}

In order to have good estimates it is as usual necessary to introduce the refinement not close to the (fixed) boundary of $Q$ but instead close to the boundary of a slightly smaller cube, and to work with several coverings with families of cubes (denoted by $q\subset q'\subset q''\subset q'''$) which are slightly enlarged versions of each other, see Figure \ref{figure:grid} for a representation.

Let $f_0(\xi):=\frac 1p\big(\mathbb{C}\xi\cdot\xi\big)^{\sfrac p2}$, for $\xi\in\R^{n{\times} n}_\sym$, with $\C$ obeying \eqref{eqassC}. Let $Q_r:=(-r,r)^n$ and let $Q:=Q_1=(-1,1)^n$.
\begin{theorem}\label{th:app}
There exist $\eta,c$ positive constants
{and a mollifier $\rho\in C_c^\infty(B(0,1);\R_+)$} such that if $u\in GSBD^p(Q_1)$
and
$\delta:=\H^{n-1}(J_u)^{1/n}$ satisfies $\delta<\eta$,
then there exist $R\in({1-\sqrt\delta},1)$, $\tilde u\in GSBD^p(Q_1)$, and $\tilde \omega\subset Q_R\subset\subset Q_1$, such that
\begin{enumerate}
 \item \label{prop:app1}
$\tilde u\in C^\infty({Q_{{1-\sqrt\delta}}})$, $\tilde u=u$ in $Q_1\setminus Q_R$, $\Hn(J_u\cap \partial Q_R)=\Hn(J_{\tilde u}\cap \partial Q_R)=0$;
 \item \label{prop:app2}
$\Hn(J_{\tilde u}\setminus J_u)\leq {c\sqrt\delta}\Hn(J_u\cap (Q_1\setminus Q_{{1-\sqrt\delta}}))$;
 \item \label{prop:app3}
{
It holds
\[
\|e(\tilde u)-\rho_{\delta}*e(u)\|_{L^p(Q_{1-\sqrt{\delta}})}
\le c\delta^s \|e(u)\|_{L^p(Q)}
\]
and for} any open set $\Omega\subset Q$ we have
 \begin{equation*}
 \displaystyle\int_\Omega f_0(e(\tilde u))dx\leq \int_{\Omega_\delta}f_0(e(u))dx+c\delta^s\int_{Q_1}f_0(e(u))dx,
 \end{equation*}
where $\Omega_\delta:=Q\cap \cup_{x\in\Omega} (x+(-3\delta,3\delta)^n)$;
  and $s\in(0,1)$ depends only on $n$ and $p$.
 \item \label{prop:app-u}
$|\tilde \omega|\leq {c\delta}\Hn(J_u\cap Q_R)$ and
 $\displaystyle\int_{Q\setminus \tilde \omega}|\tilde u-u|^pdx\leq c\delta^p\int_Q|e(u)|^pdx$;
 \item \label{prop:app-psif}
If $\psi\in\Lip(Q;[0,1])$, then
\[
\int_Q \psi f_0(e(\tilde u))dx
\le \int_Q \psi f_0(e( u))dx +c\delta^{s}{(1+\Lip(\psi))}\int_Q |e(u)|^p dx.
\]
{\item \label{prop:app-lp}
If in addition $u\in L^p(Q)$, then for $\Omega\subset Q$,
\begin{equation*}
\|\tilde u\|_{L^p(\Omega)}\le \| u\|_{L^p(\Omega)}
+ c\delta^{{\frac{1}{2p}}}(\|u\|_{L^p(Q)}+ \|e(u)\|_{L^p(Q)}).
\end{equation*}
}
\end{enumerate}
The constant $c$ depends on $n$, $p$, and $\C$.
\end{theorem}

\begin{proof} 
We let $\eta:=1/(2\cdot 8^n c_*)$, where $c_*$ is the constant entering (\ref{eq:stimom}). We can assume without loss of generality that $c_*\ge 1$.

Let $N:=[1/\delta]$, so that $(-N\delta,N\delta)^n\subseteq Q$.
For $i=0,\dots,N-1$ we let $Q^i:=(-(N-i)\delta,(N-i)\delta)^n$
and $C^i:=Q^i\setminus Q^{i+1}$ ($C^{N-1}=Q^{N-1}$). Up to a small translation of the $Q^i$'s we can assume that $J_u$ does not intersect the boundaries $\partial Q^i$ and that almost every point $y\in \partial Q^{i}$ 
is a Lebesgue point for $e(u)$, in the sense that
\begin{eqnarray}\label{eq:nochargei}
&\H^{n-1}(J_u\cap\partial Q^{i})=0, \\
& \displaystyle\lim_{r\to0}\fint_{B_{r}(y)}|e(u)-e(u)(y)|^pdx=0,\qquad\Hn\textrm{-a.e. }y\in\partial Q^i,
\end{eqnarray}
for every $i=0,\dots,N-1$. 

We choose {$i_0\in \N\cap[1,{1/\sqrt\delta-3}]$}
such that (see Lemma \ref{lemmachoice}, and observe that
$(N-[1/\sqrt{\delta}]+1)\delta\ge 1-\sqrt{\delta}$) we have both
\begin{equation}\label{eq:estimcrown}
\begin{cases}\displaystyle
\int_{C^{i_0}\cup C^{i_0+1}} |e(u)|^pdx   \le {8}\sqrt{\delta}\int_{Q\setminus Q_{{1-\sqrt\delta}}} |e(u)|^p dx
\\[2mm]
\displaystyle
\H^{n-1}(J_u\cap (C^{i_{0}}\cup C^{i_{0}+1}))
\le {8}\sqrt{\delta}\H^{n-1}(J_u\cap (Q\setminus Q_{{1-\sqrt{\delta}}}))
\end{cases}
\end{equation}
{if $\delta$ is small enough.}

\begin{figure}
\centerline{\includegraphics[width=6cm]{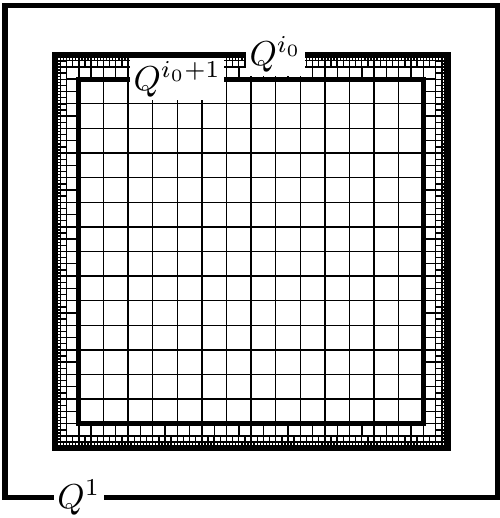} \hskip5mm \includegraphics[width=8cm]{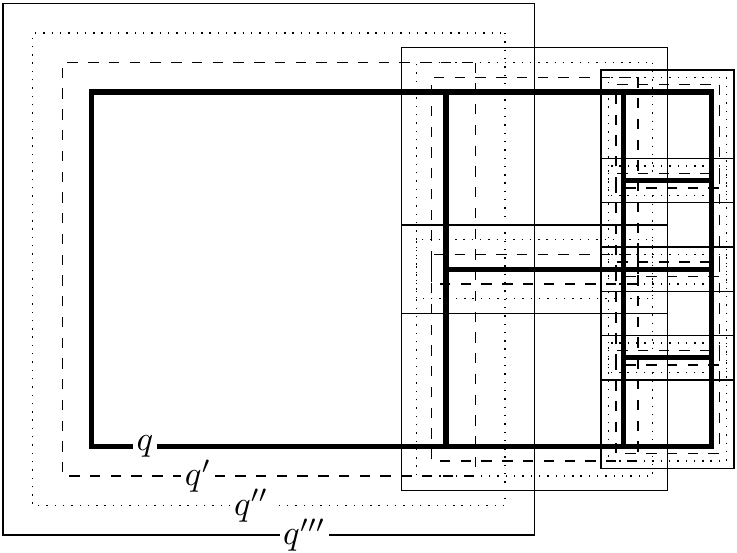}}
 \caption{Sketch of the decomposition of {$Q^{i_0}$} into disjoint squares. Left: global decomposition. The refinement takes place in $C^{i_0}=Q^{i_0}\setminus Q^{i_0+1}$. Right: blow-up showing a few of the cubes $q$ and the corresponding enlarged cubes $q'$, $q''$, $q'''$. See text.}
 \label{figure:grid}
\end{figure}

Next, we cover $Q^{i_0}$ by disjoint cubes, up to a null set. First we subdivide
$Q^{i_0+1}$ into cubes $z+(0,\delta)^n$, $z\in \delta\Z^n$. Then we divide the crown $C^{i_0}$ into dyadic
slabs 
\[
S_k := ( -(N-i_0-2^{-k})\delta,(N-i_0-2^{-k})\delta)^n
\setminus
( -(N-i_0-2^{-k+1})\delta,(N-i_0-2^{-k+1})\delta)^n,
\]
$k=1,\dots,\infty$ and each slab $S_k$ into $\sigma_k$
cubes of the type $z+(0,\delta 2^{-k})^n$, $z\in 2^{-k} \delta\Z^n$,
so that
$\sigma_k\le C 2^{k(n-1)}/\delta^{n-1}$. Here and henceforth $C$ will denote a dimensional constant.
We denote by $\mathcal{W}$ the collection of these cubes
and $\mathcal{W}_0\subset\mathcal{W}$ the cubes of size $\delta$, which cover
the central cube $Q^{i_0+1}$ (see Figure \ref{figure:grid}).
For $q\in\mathcal{W}$, we
let $q'$, $q''$, $q'''$, be respectively the cubes
with same center and dilated by $7/6$, $4/3$, $3/2$.
In particular, the cubes $q'''$ have a finite overlap; also, one has
\[
\bigcup_{q\in \mathcal{W}\setminus \mathcal{W}_0} q'''\subset C^{i_0}\cup C^{i_0+1}.
\]
Given $q\in\mathcal{W}$, we say that $q$ is ``good'' if
\begin{equation}\label{eq:goodcube}
\H^{n-1}(q'''\cap J_u) \le \eta\delta_q^{n-1}
\end{equation}
where $\delta_q$ is the size of the edge of the cube ($\delta_q=\delta$
if $q\in\mathcal{W}_0$ and
$\delta 2^{-k}$ if $q\subseteq S_k$, $k\ge 1$). We say that $q$ is ``bad''
if~\eqref{eq:goodcube} is not satisfied. Observe in particular
that since by definition, $\delta^n=\H^{n-1}(J_u)\le \eta\delta^{n-1}$, 
each $q\in \mathcal{W}_0$ is good. On the other hand, 
there are at most (since the cubes $q'''$ can overlap with their
neighbours)
\[C(2^{k(n-1)}/\delta^{n-1})\H^{n-1}(J_u\cap (S_{k-1}\cup S_k\cup S_{k+1}))/\eta\]
bad cubes in the slab $S_k$, with a  total perimeter  bounded by
\[
\frac{C}{\eta}\H^{n-1}(J_u\cap (S_{k-1}\cup S_k\cup S_{k+1})).
\]
Hence the 
total perimeter of the union $\mathcal{B}$  of the bad cubes is bounded
by 
\begin{equation}\label{eq:newjumpset}
\H^{n-1}(\partial^*\mathcal{B})\le \frac{C}{\eta}\H^{n-1}(J_u\cap (C^{i_0}\cup C^{i_0+1}))
\le 
\frac{C}{\eta}\sqrt{\delta}\H^{n-1}(J_u\cap (Q\setminus Q_{{1-\sqrt{\delta}}}))
\end{equation}
thanks to~\eqref{eq:estimcrown}. Observe also that $\mathcal{B}\subset C^{i_0}$,
with 
\[
 |\mathcal{B}| \le C\frac{\delta^{3/2}}{\eta}\H^{n-1}(J_u\cap (Q\setminus Q_{{1-\sqrt{\delta}}})).
\]
We enumerate the good cubes, denoted $(q_i)_{i=1}^\infty$, {and we assume that $\mathcal W_0=\bigcup_{i\leq N_0} q_i$ for some $N_0\in \N$}.
We construct a partition of unity associated to the ``good'' cubes $(q_i)$, so that
 $\p_i\in C^\infty(Q^{i_0}\setminus \overline{\mathcal{B}};[0,1])$,
 $\p_i=0$ on $Q^{i_0}\setminus q'_i\setminus {\mathcal{B}}$, 
 $|\nabla \p_i|\le C/\delta_{q_i}$ and $\sum_i \p_i=1$ on $Q^{i_0}\setminus\mathcal{B}$, the sum being locally finite.
To do this, we first choose for each $q_i'$ a function $\tilde \varphi_i\in C^\infty_c(q'_i;[0,1])$ with 
$\tilde{\varphi}_i=1$ on $q_i$ and $|\nabla\tilde{\varphi}_i|\leq {C}/\delta_{q_i}$, then in $Q^{i_0}\setminus \mathcal{B}$ we set $\varphi_i:=\tilde\varphi_i/(\sum_j\tilde{\varphi_j})$. 
{We recall that the sum runs over all good cubes and} stress that the functions $\p_i$ are not defined in $\mathcal B$.

Thanks to Prop.~\ref{th:prop3b},
for each good cube $q_i$ one can find a set $\omega_i\subset q''_i$ with
$|\omega_i|\le c_*\delta_{q_i}\H^{n-1}(J_u\cap q'''_i)
{\le c_*\eta\delta_{q_i}^{n}}$, and an affine function
$a_i$ with $e(a_i)=0$ such that
\begin{equation}\label{eq:ccf}
\int_{q''_i\setminus\omega_i}|u-a_i|^p{dx}\le c \delta_{q_i}^p \int_{q'''_i} |e(u)|^p dx,
\end{equation}
\begin{equation}\label{eq:stimi2}
\int_{q''_i\setminus\omega_i} |u-a_i|^{np/(n-1)}{dx}
\le c\delta_{q_i}^{n(p-1)/(n-1)}\left(\int_{q'''_i} |e(u)|^p dx\right)^{n/(n-1)}.
\end{equation}
Moreover, given a symmetric mollifier $\rho$ with support in $B_{1/6}$,
if one lets
\begin{equation}
u_i := \rho_{\delta_{q_i}}*(u\chi_{q''_i\setminus\omega_i} + a_i\chi_{\omega_i}),
\end{equation}
then
\begin{equation}\label{eq:estimmolif}
\int_{q'_i} |e(u_i)-e(u)*\rho_{\delta_{q_i}}|^p dx\le
c\left(\frac{\H^{n-1}(J_u\cap q'''_i)}{\delta_{q_i}^{n-1}}\right)^{\bar p}
\int_{q'''_i} |e(u)|^p dx
\end{equation}
where $c$ depends on $\rho,n,p$.
{Observe in addition that (the mollifier being symmetric,
one has $\rho_{\delta_{q_i}}*a_i=a_i$):
\begin{multline}\label{eq:stimi3}
\int_{q'_i} |u_i-a_i|^p dx = \int_{q'_i}|\rho_{\delta_{q_i}}*((u-a_i){\chi_{q''_i\setminus \omega_i}})|^p dx 
\\\le
\int_{q''_i\setminus\omega_i} |u-a_i|^p dx \le
c \delta_{q_i}^p\int_{q'''_i} |e(u)|^p dx
\end{multline}
thanks to~\eqref{eq:ccf}.
}

Notice that if $q_i$ and $q_j$ are touching,
one can estimate the distance between $a_i$ and $a_j$.
Using (\ref{eq:ccf}) on the two squares $q_i''$ and $q''_j$ we find
\begin{multline*}
\|a_i-a_j\|_{L^{p}(q''_i\cap q''_j\setminus (\omega_i\cup\omega_j))}
\le\|a_i-u\|_{L^{p}(q''_i\setminus \omega_i)}+
\|a_j-u\|_{L^{p}(q''_j\setminus \omega_j)}\\
\le c \delta_{q_i} \|e(u)\|_{L^p(q_i''')} + c \delta_{q_j} \|e(u)\|_{L^p(q_j''')}.
\end{multline*}
We now observe that by the properties of the grid, if $q''_i\cap q''_j\ne \emptyset$ then necessarily 
$|q''_i\cap q''_j|\ge 4^{-n}\max\{|q_i|, |q_j|\}$ (the critical case is the one with two squares whose side lengths differ by a factor of 2, and which share a corner).
By the choice of $\eta$, since $q_i$ and $q_j$ are good we obtain $|\omega_i|\le \frac12 8^{-n}|q_i|$, and the same for $j$. 
Therefore
$|\omega_i\cup\omega_j|\le |q''_i\cap q''_j|/4$. 
Since $a_i-a_j$ is affine 
(see for example \cite[Lemma 4.3]{ContiFocardiIurlano2016}, which generalizes immediately to parallelepipeds)
\begin{equation}\label{eq:stimij}
\|a_i-a_j\|_{L^{np/(n-1)}(q''_i\cap q''_j)}\le c\delta_{q_i}^{(p-1)/p}\|e(u)\|_{L^p(q'''_i\cup q'''_j)}.
\end{equation}
This is the point which motivates the choice of $\eta$.

We then define
\begin{equation*}
 \tilde u = 
 \begin{cases}
 \sum_i u_i\p_i  & \text{ in $Q^{i_0}\setminus\mathcal B$}\\
 u & \text{ in $\mathcal B\cup (Q\setminus Q^{i_0})$.}
 \end{cases}
\end{equation*}
and observe that  $\tilde u$ is smooth in $Q^{i_0}\setminus\overline{\mathcal B}$, with
\begin{equation}\label{eq:1}
e(\tilde u) = \sum_i \p_ie(u_i)+\sum_i \nabla \p_i\odot u_i. 
\end{equation}
Since $\sum_i\nabla \p_i = 0$, we write $\nabla \p_i=-\sum_{j\sim i}\nabla \p_j$
(where $i\sim j$ if $i\ne j$ and $q'_i\cap q'_j\neq \emptyset$).
Hence for each $x\in q'_i$,
\begin{equation}\label{eq:pointbound}
\sum_l(\nabla\p_l\odot u_l)(x)=(\nabla\p_i\odot u_i)(x) +
\sum_{j\sim i} (\nabla\p_j\odot u_j)(x) 
= \sum_{j\sim i}(\nabla\p_j\odot (u_j-u_i))(x) .
\end{equation}

Let now $\Omega\subset Q$ be open, and define 
$\Omega_\delta$ as a $3\delta$-neighbourhood in the $\|\cdot\|_\infty$ norm, in the sense that
\begin{equation*}
 \Omega_\delta := \{x\in Q: \exists y\in \Omega, |x_i-y_i|<3\delta \text{ for } i=1,\dots, n\} =
 Q\cap\bigcup_{x\in \Omega} (x+3(-\delta,\delta)^n).
\end{equation*}
The key property of this neighbourhood, which motivates the choice of the factor $3$, is the fact that {if} $q_i'\cap \Omega\ne\emptyset$ 
then $q_i'''\subset\Omega_\delta$ and, additionally, for any $j$ with $q_i\sim q_j$ one has $q_j'''\subset\Omega_\delta$.
In the following, all constants will not depend on the choice of $\Omega$.

{We first introduce
$\tilde Q := Q^{i_0+1}\setminus \bigcup_{i>N_0} q'_i$
and start with estimating~\eqref{eq:1} in $\tilde Q$, where
all mollifications are at scale $\delta$.
In particular, if $x\in\tilde Q$,
then all cubes $q_j$ appearing in the right-hand side of~\eqref{eq:pointbound}
(with a non-vanishing term) are of size $\delta$.
}
 
For any two good cubes $q_i\sim q_j$, $i,j\le N_0$, we have
\begin{multline*}
\|u_i-u_j\|_{L^p(q'_i\cap q'_j)}=
\|\rho_{\delta}*(u\chi_{q''_i\setminus \omega_i}+a_i\chi_{\omega_i}-
u\chi_{q''_j\setminus \omega_j}-a_j\chi_{\omega_j})\|_{L^p(q'_i\cap q'_j)}
\\
\le
\|
(u-a_i)\chi_{\omega_j\setminus \omega_i}
-(u-a_j)\chi_{\omega_i\setminus \omega_j}
+ (a_i-a_j)\chi_{\omega_i\cup \omega_j}
\|_{L^p(q''_i\cap q''_j)}
\\\le
\|u-a_i\|_{L^{\frac{pn}{n-1}}(q''_i{\setminus \omega_i})}|\omega_j|^{\frac{1}{np}}
+\|u-a_j\|_{L^{\frac{pn}{n-1}}(q''_j{\setminus \omega_j})}|\omega_i|^{\frac{1}{np}}
\\+\|a_i-a_j\|_{L^{\frac{pn}{n-1}}(q''_i\cap q''_j)} |\omega_i\cup\omega_j|^{\frac{1}{np}}.
\end{multline*}

{Recalling~\eqref{eq:stimi2}, \eqref{eq:stimij}, \eqref{eq:pointbound}, {and $|\nabla \varphi_i|\leq c/\delta$,} we obtain
\begin{equation*}
\left\|\sum_i\nabla\p_i\odot u_i\right\|_{L^p(\Omega\cap\tilde Q)}
\le c 
\sum_{i: q_i'''\subset \Omega_\delta, i\le N_0}
{\left(\frac{|\omega_i|^{1/n}}{\delta}\right)^{1/p}}
\sum_{j\sim i,j\le N_0} \|e(u)\|_{L^p(q_j''')}.
\end{equation*}
}
Since 
$|\omega_i|\le c_*\delta^{n+1}$ for $i\le N_0$, we see that the factor is bounded by {$\delta^{1/(np)}$} and 
\begin{equation}
\left\|{\sum_{i=1}^{N_0} }\nabla \p_i\odot u_i\right\|_{L^p(\Omega\cap \tilde Q)}
\le c{\delta^{1/(np)}}\|e(u)\|_{L^p(\Omega_\delta{\cap Q^{i_0}})}.
\label{eq:1right}
\end{equation}

{It follows from~\eqref{eq:1} and~\eqref{eq:1right} that
\begin{multline}\label{eq:40}
\|e(\tilde u)-\rho_\delta*e(u)\|_{L^p(\Omega\cap \tilde Q)}
\\ \le \left\|\sum_{i=1}^{N_0} \p_i(e(u_i)-e(u)*\rho_\delta)\right\|_{L^p(\Omega\cap\tilde Q)} 
+
c{\delta^{1/(np)}}\|e(u)\|_{L^p(\Omega_\delta {\cap Q^{i_0}})}
\\ \le
\sum_{i\le N_0, q_i'''\subset\Omega_\delta}\|e(u_i)-e(u)*\rho_\delta\|_{L^p(q'_i)} +
c{\delta^{1/(np)}}\|e(u)\|_{L^p(\Omega_\delta {\cap Q^{i_0}})}
\\ \le
 c\sum_{i\le N_0, q_i'''\subset\Omega_\delta}\left(\frac{\H^{n-1}(J_u\cap q'''_i)}{\delta^{n-1}}\right)^{\frac{\bar p}{p}}
\|e(u)\|_{L^p(q'''_i)}
+c{\delta^{1/(np)}}\|e(u)\|_{L^p(\Omega_\delta {\cap Q^{i_0}})}
\\
\le c\delta^s \|e(u)\|_{L^p(\Omega_\delta\cap Q^{i_0})}
\end{multline} 
where $s:=\min\{\bar p/p,1/(np)\}$
and we have used~\eqref{eq:estimmolif}. We deduce the first assertion
in Property~\ref{prop:app3}.~by choosing $\Omega=Q_{1-\sqrt{\delta}}$.
In addition, it follows that
\begin{multline}\label{eq:4}
\Big(\int_{\Omega\cap \tilde Q}f_0(e(\tilde u))dx\Big)^{1/p}\leq
\Big(\int_{\Omega\cap \tilde Q}f_0(e(u)*\rho_\delta)dx\Big)^{1/p}
+ c\delta^s \|e(u)\|_{L^p(\Omega_\delta\cap Q^{i_0})}
\\
\leq \Big(\int_{\Omega_\delta\cap \tilde Q}f_0(e(u))dx\Big)^{1/p}
+ c\delta^s \|e(u)\|_{L^p(\Omega_\delta\cap Q^{i_0})}
\\
\leq (1+c\delta^s)\Big(\int_{\Omega_\delta\cap Q^{i_0}}f_0( e(u))dx\Big)^{1/p}.
\end{multline}
}
Observing that
{for $\delta\le \eta\le 1$, 
$(1+c\delta^{ s})^p\leq  1+ pc(1+c)^{p-1} \delta^{ s}$,}
we end up with the estimate
\begin{equation}\label{eq:estimcenter}
\int_{\Omega\cap\tilde Q}f_0(e(\tilde u))dx
\le (1+c\delta^{s})\int_{\Omega_\delta\cap Q^{i_0}}f_0(e(u))dx .
\end{equation}

We now turn to the boundary layer and estimate $\int_{\Omega\cap Q^{i_0}\setminus \tilde Q} f_0(e(\tilde u))dx$.
This is done in a similar way, however less precise, and we only
find that
\begin{equation}\label{eq:crownep}
\int_{\Omega\cap Q^{i_0}\setminus \tilde Q} f_0(e(\tilde u)) dx
\le  c\int_{\Omega_\delta\cap (C^{i_0}\cup C^{i_0+1})} f_0(e(u)) dx.
\end{equation}
Indeed, for now $i>N_0$ (that is, $q_i$ a good cube at
scale $\delta 2^{-k}$ for some $k\ge 1$), one writes first that
(using again that $\rho_{\delta_{q_i}}*a_i=a_i$ since $\rho$ is even),
\begin{multline*}
\|e(u_i)\|_{L^p(q'_i)} =
\|e(u_i-a_i)\|_{L^p(q'_i)} =
\|e(\rho_{\delta_{q_i}}*(u\chi_{q_i''\setminus\omega_i}+a_i\chi_{\omega_i}-a_i))\|_{L^p(q'_i)}
\\
\le \frac{\|\nabla\rho\|_{L^1}}{\delta_{q_i}}
\|u-a_i\|_{L^p(q_i''\setminus\omega_i)}
\le C\|e(u)\|_{L^p(q'''_i)},
\end{multline*}
thanks to~\eqref{eq:ccf}.
Hence,
to show~\eqref{eq:crownep} it remains to estimate
$\|\sum_i \nabla\p_i\odot u_i\|_{L^p(q'_i)}$ for $i> N_0$.
As before we observe that this is bounded by
\[
\sum_{j\sim i} \|\nabla \p_j\odot (u_i-u_j)\|_{L^p(q'_j\cap q'_i)}
\]
and thanks to the fact that $|\nabla\p_i|\le C/\delta_{q_i}$ and,
when $j\sim i$, 
$\delta_{q_j}\in \{(1/2)\delta_{q_i},\delta_{q_i},2\delta_{q_i}\}$),
each term in the sum is bounded by 
\[
\frac{C}{\delta_{q_i}}\|u_i-a_i\|_{L^p(q'_i)}+ 
\frac{C}{\delta_{q_j}}\|u_j-a_j\|_{L^p(q'_j)}+ 
\frac{C}{\delta_{q_i}}\|a_i-a_j\|_{L^p(q'_i\cap q'_j)}.
\]
Thanks 
{to~\eqref{eq:stimi3}}
and~\eqref{eq:stimij}, this is bounded
by $\|e(u)\|_{L^p(q'''_i\cup q'''_j)}$ and~\eqref{eq:crownep} follows.

Hence using \eqref{eq:estimcrown} 
and $\tilde u=u$ in $\mathcal B\cup Q\setminus Q^{i_0}$ we get
\begin{equation}\label{eq:estimenergy}\int_{\Omega}f_0(e(\tilde u))dx\leq (1+c\delta^s)\int_{\Omega_\delta}f_0(e(u))dx
{+ c\sqrt{\delta}\int_{Q\setminus Q_{{1-\sqrt{\delta}}}} f_0(e(u)) dx.}
\end{equation}
{Using $s<1/2$, we deduce Property \ref{prop:app3}.}

Let now $\psi\in\Lip(Q;[0,1])$. Then
\[
 \int_Q \psi f_0(e(\tilde u))dx = \int_Q \int_0^{\psi(x)} f_0(e(\tilde u))dx  dt
  = \int_0^1  \int_{\{x: t<\psi(x)\}} f_0(e(\tilde u))dx  dt.
\]
If $\Omega=\{x: t<\psi(x)\}$, then 
$\Omega_\delta\subset \{x: t<\psi(x)+c_\psi \delta\}$, where $c_\psi=3 n^{1/2} \Lip(\psi)$.
Therefore \eqref{eq:estimenergy} implies
{\begin{multline*}
 \int_Q \psi f_0(e(\tilde u))dx 
  \\
\le \int_0^1
\Big(  (1+c\delta^s) \int_{\{x: t<\psi(x) +c_\psi\delta\}} f_0(e( u))dx  
+ c\sqrt{\delta}\int_{Q\setminus Q_{{1-\sqrt{\delta}}}} f_0(e(u)) dx 
 \Big)dt\\
  = (1+c\delta^s) \int_Q   (\psi(x)+c_\psi\delta) f_0(e( u))dx  
+ c\sqrt{\delta}\int_{Q\setminus Q_{{1-\sqrt{\delta}}}} f_0(e(u)) dx.
\end{multline*}
}
This proves Property~\ref{prop:app-psif}.

We finally estimate the distance between $\tilde u$ and $u$ outside of 
{$\tilde \omega:= \bigcup_i\omega_i\setminus \calB$.}
We find by \eqref{eq:ccf}
{and~\eqref{eq:stimi3}}:
\begin{multline}\label{eq:6}
\int_{Q\setminus \tilde{\omega}}|\tilde u-u|^pdx\leq c\sum_i \int_{q'_i\setminus\omega_i}|u_i-u|^pdx\leq\\
\leq c\sum_i\int_{q'_i}|u_i-a_i|^pdx+c\sum_i\int_{q'_i\setminus\omega_i}|u-a_i|^pdx
\leq c\delta^p\int_Q|e(u)|^pdx.
\end{multline}
This proves Property~\ref{prop:app-u}, together with the
observation\footnote{\cite[Prop.~6]{ChambolleContiFrancfort2016} shows that, in fact, one
could ensure that $|\omega_i|\le c\H^{n-1}(J_u\cap q_i''')^{n/(n-1)}$. With
our choice of $\delta$, this improves the
inequality to $|\omega_i|\le c\delta^{n/(n-1)}\H^{n-1}(J_u\cap q_i''')$.
Hence the first point in Prop.~\ref{prop:app-u} could be
improved, in fact, to $|\tilde{\omega}|\le c\delta^{n/(n-1)}
\H^{n-1}(J_u\cap Q_R)\approx \H^{n-1}(J_u\cap Q_R)^{n/(n-1)}$.}
that $|\omega_i|\leq c_*\delta\H^{n-1}(J_u\cap q_i''')$.

One has moreover $J_{\tilde u}\cap Q^{i_0}\subseteq \partial^*\mathcal{B}{\cup
{(J_u\setminus Q^{i_0+1})}}$, {and $J_{\tilde u}\setminus J_u
\subseteq \partial^*\mathcal{B}$,} which is bounded
by~\eqref{eq:newjumpset}.

For $y\in \partial Q^{i_0}$ one also obtains (letting $\delta_k:=\delta2^{-k}$) that
\[
\int_{B_{\delta_k}(y)\setminus\tilde \omega} |\tilde u-u|^p {dx}\le C\delta^p 2^{-kp}
\int_{B_{\delta_{k-1}}(y)}|e(u)|^p dx,
\]
while $|\tilde \omega\cap B_{\delta_k}(y)|\le C\delta_k (1+1/\eta)\H^{n-1}(J_u\cap B_{\delta_{k-1}}(y))$. 
Hence we have for every $\eps>0$ and for $\Hn$-a.e. $y\in\partial Q^{i_0}$
\begin{multline*}
\frac{1}{\delta_k^n}\left|\left\{|\tilde{u}-u|>\eps \right\}\cap B_{\delta_k}(y)\right|\leq\\
\leq\frac{|\tilde \omega\cap B_{\delta_k}(y)|}{\delta_k^n}+\frac{1}{\eps}\left(\fint_{B_{\delta_k}(y)\setminus\tilde \omega} |\tilde u-u|^pdx\right)^{1/p}\\
\leq\frac{c\Hn(J_u\cap B_{\delta_k}(y))}{\delta_k^{n-1}}+c\delta_k\left(\fint_{B_{\delta_{k-1}}(y)}|e(u)|^pdx\right)^{1/p},
\end{multline*}
with the last line which vanishes
{as $\delta_k\downarrow 0$} thanks to \eqref{eq:nochargei}.
Hence one can deduce that
the trace of $\tilde u$ on $\partial Q^{i_0}$ coincides with
the inner trace of $u$, showing that $\tilde u$ is $GSBD$ when extended
with the value of $u$ out of $Q^{i_0}$, with $J_{\tilde u}\setminus {Q^{i_0}}=J_u$.
\smallskip
In particular we obtain that 
\[ \tilde u\in GSBD^p(Q)\cap C^\infty(Q^{i_0+1}), \quad \textup{and}\quad
\tilde u\in C^\infty( (-({1-\sqrt{\delta}}),{1-\sqrt{\delta}})^n).
\]

{Assume eventually that $u\in L^p(Q{;\R^n})$,
and let us show also that we can ensure also Property~\ref{prop:app-lp}.~in
this case. Let $\Omega\subset Q$. 
One now has thanks to~\eqref{eq:6}
\begin{equation}\label{eq:firstlp}
\|\tilde u\|_{L^p(\Omega)}\le  \|u\|_{L^p(\Omega\setminus\tilde\omega)}
+  c\delta \|e(u)\|_{L^p(Q)} +\|\tilde u\|_{L^p(\Omega\cap \tilde \omega)}
\end{equation}
and we need to estimate the last term, starting from the fact
that
\begin{equation}\label{eq:prop61}
\int_{\Omega\cap \tilde\omega} |\tilde u|^pdx \le \sum_{i=1}^\infty \int_{\Omega\cap q'_i\cap\tilde\omega} |u_i|^pdx.
\end{equation}
For each $i\ge 1$,
\begin{multline}\label{eq:prop62}
\int_{\Omega\cap q'_i\cap\tilde\omega} |u_i|^p dx
\le c\int_{\Omega\cap q'_i\cap\tilde\omega} |u_i-a_i|^pdx + c
\int_{\Omega\cap q'_i\cap\tilde\omega} |a_i|^p dx
\\
\le c\delta_{q_i}^p \int_{q'''_i} |e(u)|^pdx + c
\int_{\Omega\cap q'_i\cap\tilde\omega} |a_i|^p dx
\end{multline}
thanks to~\eqref{eq:stimi3}.
We now consider $q_i$ such that $q'_i\cap\Omega\neq\emptyset$.
Notice that 
\[
|\tilde\omega\cap q'_i|
\le c_*\left[\delta_{q_i}\H^{n-1}(J_u\cap q_i''')+\sum_{j\sim i}\delta_{q_j}\H^{n-1}(J_u\cap q_j''')\right]
\le c_*\min\{3^n\delta^{n+1},6^n \eta\delta_{q_i}^n\},
\]
which is small if $\eta$ is small enough. 
In this case, thanks to Lemma~\ref{lemmarigid} (and more precisely
its consequence~\eqref{eq:estima})\footnote{this requires an additional
condition on $\eta$ which might need to be reduced.}
\[
\int_{\Omega\cap q'_i\cap\tilde\omega} |a_i|^p dx
\le c \frac{|\tilde\omega\cap q'_i|}{|q_i|} \int_{q_i\setminus \tilde\omega_i} |a_i|^p dx
\le c \frac{|\tilde\omega\cap q'_i|}{|q_i|} \int_{q_i\setminus \tilde\omega_i} {(}|u-a_i|^p
+ |u|^p{)} dx.
\]
The first term in the integral
is estimated as usual with $\delta_{q_i}^p\int_{q'''_i}|e(u)|^pdx$,
while, for the second, we use that
\[
\frac{|\tilde\omega\cap q'_i|}{|q_i|} 
\le 
\begin{cases} 3^nc_* \delta  
& \textup{ if } i\le N_0,\\
6^n c_*\eta 
& \textup{ if } i>N_0
\end{cases}
\]
to deduce
\[
\sum_i \int_{\Omega\cap q'_i\cap\tilde\omega} |a_i|^p dx
\le c\delta^p\int_{Q}|e(u)|^p dx + c\delta\int_{\Omega_\delta\cap Q^{i_0+1}}|u|^p dx
+ c\eta \int_{\Omega_\delta\cap C^{i_0}} |u|^p dx.
\]
Using~\eqref{eq:prop61}, \eqref{eq:prop62}, we find
\begin{equation*}
\int_{\Omega\cap\tilde\omega}|\tilde u|^p dx\le c\delta^p\int_Q|e(u)|^pdx
+ c\delta\int_{\Omega_\delta}|u|^p dx
+ c\eta\int_{\Omega_\delta \cap C^{i_0}}|u|^p dx.
\end{equation*}
To estimate the last term, there are two possible strategies: one
is to send $\eta\to 0$ as $\delta\to 0$, the other which
we adopt here is to require, when selecting the index $i_0$, that
in addition to~\eqref{eq:estimcrown} it satisfies:
\[
\int_{C^{i_0}} |u|^pdx   \le 8\sqrt{\delta}\int_{Q\setminus Q_{1-\sqrt\delta}} |u|^p dx,
\]
which can be ensured in the same way. We deduce in this
case that
\begin{equation}\label{eq:lastlp}
\int_{\Omega\cap\tilde\omega}|\tilde u|^p dx\le c\delta^p\int_Q|e(u)|^pdx
+ c\delta\int_{\Omega_\delta}|u|^p dx
+ c\sqrt{\delta}\int_{Q\setminus Q_{1-\sqrt{\delta}}}|u|^p dx.
\end{equation}}
{Hence we obtain Property~\ref{prop:app-lp}.~from~\eqref{eq:firstlp} and~\eqref{eq:lastlp}.}

\end{proof}

\begin{lemma}\label{lemmachoice}
 Let $a_i\ge0$, $b_i\ge 0$ be such that
 \begin{equation*}
  \sum_{i=1}^k a_i\le A \hskip5mm\text{and}\hskip5mm\sum_{i=1}^k b_i\le B\,.
 \end{equation*}
 Then there is $j\in \{1,\dots, k\}$ such that
 \begin{equation*}
   a_j\le \frac2k A \hskip5mm\text{and}\hskip5mm b_j\le\frac2k B\,.
 \end{equation*}
\end{lemma}
\begin{proof}
 We write
 \begin{equation*}
  \sum_{i=1}^k \frac{a_i}{A} +\frac{b_i}{B}\le 2
 \end{equation*}
 and choose $j$ as one index for which the summand is minimal. Notice that if $A=0$ then all $a_i$ are also 0.
\end{proof}

{
\begin{lemma}\label{lemmarigid}
Let $q\subset\R^n$ be a cube, and $\omega\subset q$ and $p\ge 1$:
there exists a constant $c$ (depending only on $n,p$) such that
\[
\int_\omega |a|^p dx \le c\frac{|\omega|}{|q|} \int_q |a|^p dx
\]
for any affine function $a:q\to \R$.
\end{lemma}
If $\theta<1$ and
$\theta q$ is the cube with same center as $q$ and sides multiplied
by $\theta$, it is possible to show by a direct computation  that for $a$ affine,
\[
\|a\|_{L^p(q)} \le \frac{1}{\theta^{n/p+1}}\|a\|_{L^p(\theta q)}.
\]
Hence one can also deduce (for $c=c(n,p,\theta)$)
\[
\int_\omega |a|^p dx \le c\frac{|\omega|}{|q|} \int_{\theta q} |a|^p dx{;}
\]
additionally if $\frac{|\omega|}{|q|}$ is small enough (so that
$\frac{|\omega|}{|q|}\left(1-c\frac{|\omega|}{|q|}\right)^{-1}\le 2\frac{|\omega|}{|q|}$), we deduce
\begin{equation}\label{eq:estima}
\int_\omega |a|^p dx \le c\frac{|\omega|}{|q|} \int_{\theta q\setminus\omega}
|a|^p dx.
\end{equation}
\begin{proof}
By translation and scaling, it is enough to show the lemma for $q=(0,1)^n$.
Then, $\|a\|_{L^p(q)}$ is a norm on the finite-dimensional space of
 affine functions. Hence there exists $c$ (depending only
on $p$) such that $\sup_{x\in q}|a(x)|\le c\|a\|_{L^p(q)}$. We conclude
by observing that $(1/|\omega|)\int_\omega |a(x)|{^p}dx \le \sup_{x\in q}|a(x)|{^p}$
for any $\omega\subset q$.
\end{proof}
}

\section{{Limit of minimizing sequences with vanishing jump}}\label{s:lim}

Let ${\cost}\ge 0$, $\beta>0$, $p>1$, $g\in L^\infty(\Omega{;\R^n})$, 
$\parametermu\geq 0$, {and let $\Omega\subset\R^n$ be a bounded, open, Lipschitz set}. For all $u\in GSBD(\Omega)$ and all {Borel sets} $A\subset\Omega$ 
let us define the functional
\begin{equation}\label{e:G} G(u,\cost,\beta,A):=\int_A f_\parametermu(e(u))dx +\cost\int_A|u-g|^pdx
+ \beta \Hn(J_u\cap {A}),
\end{equation}
where
\begin{equation}
\label{eqdeffintro}
 f_\parametermu(\xi):=\frac 1p\left(\big(\mathbb{C}\xi\cdot\xi+\parametermu\big)^{\sfrac p2}-\parametermu^{\sfrac p2}\right)
 \end{equation}
and $\C$ is a symmetric linear map from $\R^{n\times n}$ to itself 
which satisfies~\eqref{eqassC}.

{We assume that $u$ is a local minimizer of $G$ in $GSBD(\Omega)$. We remark that, if $\cost>0$, then 
a global minimizer exists} by \cite[Theorem 11.3]{gbd}.

Let us introduce the following homogeneous version of $G$   
\[ 
G_0(u,\cost,\beta,A):=\int_A f_0(e(u))dx +\cost\int_A|u|^pdx+ 
\beta \Hn(J_u\cap {A}),
\]
which will be useful to establish the decay estimate and the density lower bound.
For open sets $A\subset\Omega$ we define also
the deviation from minimality 
\[
\Psi_0(u,\cost,\beta,A):=G_0(u,\cost,\beta,A)-\Phi_0(u,\cost,\beta,A),
\] 
where 
\begin{equation}\label{e:phi0}
\Phi_0(u,\cost,\beta,A):=\inf\{G_0(v,\cost,\beta,A):\ v\in GSBD(\Omega), \ \{v\ne u\}{\subset\hskip-0.125cm\subset} A\}.
\end{equation}

{The minimality of the limit $u$ of energy-minimizing sequences $u_h$ with vanishing jump sets is established similarly to \cite[Prop. 3.3]{ContiFocardiIurlano_exist}. Any competitor of $u$ is modified close to the boundary of the domain in order to become a competitor of $u_h-a_h$ for $G_0$. Estimating in $L^p$ the symmetric gradient of the transition, one obtains a term of the form $|u_h-a_h-u|^p$, which is known to vanish only a.e. and not in $L^1$. Therefore an intermediate interpolation step which passes through $\tilde u_h-a_h$ becomes necessary. The functions $\tilde u_h$, constructed via Theorem \ref{th:app}, are smooth away from a small neighborhood of the boundary, whose size decreases as $h\to\infty$. This permits to perform the entire construction in $W^{1,p}$ and greatly simplifies the computations.}

{The following theorem is the generalization to arbitrary dimension of \cite[Prop. 3.4]{ContiFocardiIurlano_exist}.} {
Notice that the proof is made a bit easier with respect to \cite[Prop. 3.4]{ContiFocardiIurlano_exist} by the fact that the approximation we employ (see Section \ref{s:app}) is more precise than the one employed in \cite{ContiFocardiIurlano_exist} (see \cite[Sections 2 and 3]{ContiFocardiIurlanoRepr}).
}

\begin{theorem}[{Convergence} and minimality]\label{t:cm} 
Let $p\in(1,\infty)$.
Let $Q_r$ be a cube, $u_h\in GSBD^p(Q_r)$ and $\cost_h\in [0,\infty)$,
 $\beta_h\in(0,\infty)$ be two sequences 
with $\cost_h\to 0$ as $h\to\infty$, and such that
\begin{equation*}
\sup_h G_0(u_h,\cost_h,\beta_h, Q_r)<\infty,\text{ and  } \lim_{h\to\infty} \Psi_0(u_h,\cost_h,\beta_h, Q_r)=
\lim_{h\to\infty} \calH^{n-1}(J_{u_h})=0\,.
\end{equation*}
Then there exist $u\in W^{1,p}(Q_r{;\R^n})$, $\overline{a}:\R^n\to\R^n$ affine with $e(\overline{a})=0$,
and a subsequence $h_j$
such that
\begin{enumerate}
	
	\item\label{propconvenerg-itrho}  for all $t\in (0,r)$
	\[ 
	\lim_{j\to\infty} G_0(u_{h_j},\cost_{h_j},\beta_{h_j},Q_t)=\int_{Q_t} \fz(e(u)) dx+
	\int_{Q_t}|\overline{a}|^pdx;
	\]
	
	\item\label{propconvenerg-itmin} for all $v\in u+W^{1,p}_0(Q_r)$ 
	\[
	\int_{Q_r} \fz(e(u))dx\le \int_{Q_r} \fz(e(v))dx;
	\]
	
	\item \label{propconvenerg-remark}
	$u_{h_j}-a_{j}\to u$ pointwise $\calL^n$-a.e. on $Q_r$ for some affine functions $a_j$, 
	{$e(u_{h_j})\to e(u)$ in $L^p(Q_t)$, ${\beta_{h_j}\calH^{n-1}(J_{u_{h_j}}\cap Q_t)}\to 0$, 
		and ${\cost_{h_j}^{\sfrac 1p}u_{h_j}}\to\overline{a}$ in $L^p(Q_t)$ for all $t\in(0,r)$.}
\end{enumerate}	
\end{theorem}
\begin{proof}
  We assume $r=1$ (wlog) and we denote $Q:=Q_1$. 
  By monotonicity, after taking a subsequence (not relabelled) we can assume that for all~$s\in (0,{1}]$, 
  \[
  \lim_{h\to\infty} G_0(u_h,\cost_h,\beta_h,Q_s)=: \Lambda(s)
  \]
  exists and is finite.
  Since $\cost_h^{1/p}u_h$ is bounded in $L^p(Q)$,
  it has a subsequence converging to some $\bar a$, weakly.
  By \cite[Prop. 2]{ChambolleContiFrancfort2016} there exist $\omega_h,a_h$ such that $|\omega_h|\le c\H^{n-1}(J_{u_h})$ and
  \[
  \int_{Q\setminus \omega_h} |u_h-a_h|^pdx \le C\int_Q |e(u_h)|^pdx 
  \]
  which is bounded. Therefore  $\cost_h^{1/p}(u_h-a_h)\chi_{Q\setminus\omega_h}\to 0$
  strongly in $L^p$. {Recalling that 
    $\cost_h^{1/p}u_h\weakto \bar a$ weakly in $L^p$, we deduce $\chi_{Q\setminus\omega_h}\cost_h^{1/p}a_h\weakto \bar a$ 
    and then $\cost_h^{1/p}a_h\weakto \bar a$, since $|\omega_h|\to0$ and $a_h$ are affine functions.
 Being the space of affine functions finite dimensional, the convergence
 $\cost_h^{1/p}a_h\to \bar a$ is in fact strong in $L^p$ and $\bar a$ is a 
linearized rigid motion.}	
Also, observe that still up to a subsequence,
\[
(u_h-a_h)\chi_{Q\setminus \omega_h}\wto u \ \textrm{ in }L^p(Q),
\]
{for some $u\in L^p(Q{;\R^n})$.}

Let $t\in (0,1]$ be a point of (left) continuity of $\Lambda$.
Let $\delta_h=\H^{n-1}(J_{u_h}\cap Q_t)^{1/n}$ and
let $\tilde u_h$, $\tilde \omega_h$ be the functions and sets
obtained applying Theorem \ref{th:app} in $Q_t$, which obey 
$(\tilde u_h-u_h)\chi_{Q_t\setminus\tilde \omega_h}\to 0$ in $L^p$ thanks to Property~\ref{prop:app-u}.~in Theorem \ref{th:app}.

We have that up to a subsequence,
$\tilde u_h-a_h\wto \tilde u$ in $W^{1,p}(Q_s)$ for every $s<t$.
In fact, $\tilde u=u$ and we do not have to extract a subsequence
at this stage. Indeed, given $\varphi\in C_c^\infty(Q_t)$,
one has
\begin{multline*}
\int_{Q_t} (\tilde u_h-a_h)\cdot\varphi \,dx
= \int_{Q_t\cap (\omega_h\cup\tilde \omega_h)} (\tilde u_h-a_h)\cdot\varphi \,dx+
\\+ \int_{Q} 
[(u_h-a_h)\chi_{Q\setminus\omega_h}]\cdot[\varphi \chi_{Q_t\setminus \tilde \omega_h}] dx+\int_{Q_t}[(\tilde u_h-u_h)\chi_{Q_t\setminus\tilde\omega_h}]\cdot\varphi\chi_{Q_t\setminus\omega_h}dx
\\ \to \int_{Q} u\cdot\varphi \,dx
\end{multline*}
as $h\to\infty$, which shows the claim.

Observe now that thanks to Property~\ref{prop:app3}.~of Theorem \ref{th:app},
any weak limit of $e(u_h)$ in $L^p(Q_t)$ must coincide with
the weak limit of $e(\tilde u_h)$, which is $e(u)$. We deduce
that $e(u_h)\wto e(u)$ in $L^p(Q_t)$ and 
in particular that
\begin{equation}\label{eq:lsc}
\int_{Q_t}f_0(e(u)) dx \le\liminf_{h\to\infty} \int_{Q_t}f_0(e(u_h)) dx.
\end{equation}
Using also that $\cost_h^{1/p}u_h\wto \bar a$ in $L^p(Q)$, we deduce that
\begin{equation}\label{eq:lsc2}
\int_{Q_t}f_0(e(u)) + |\bar a|^p dx \le \Lambda(t).
\end{equation}

Consider now $v\in W^{1,p}(Q_t{;\R^n})$ with $\{v\neq u\}\subset\subset
Q_t$. Let $\psi\in C_c(Q_t)$ be a Lipschitz cut-off
with
$\{0<\psi<1\}\subset \{v=u\}\cap Q_{t'}$ for some $t'<t$. 
If $h$ is large enough, one also
has that $\tilde u_h\in W^{1,p}(Q_{t'}{;\R^n})$. We let
$v_h := (\tilde u_h-a_h)(1-\psi)+\psi v$. Then, for large $h$ one has  $v_h\in W^{1,p}(Q_{t'}{;\R^n})$ and $\{v_h+a_h\neq u_h\}\subset\subset Q_t$, so that by definition of $\Psi_0$ we have
	\begin{equation}\label{eq:uhvh}
	G_0(u_h,\cost_h,\beta_h,Q_t)
	\le \Psi_0(u_h,\cost_h,\beta_h,Q_1) + G_0(v_h+a_h,\cost_h,\beta_h, Q_t),
	\end{equation}	
	which we write as
	\begin{multline}\label{eq:Qt}
	\int_{Q_t} (f_0(e(u_h))  + \cost_h |u_h|^p) dx+\beta_h\Hn(J_{u_h}\cap Q_t)\\
\leq 
	 \int_{Q_t} (f_0(e(v_h))  + \cost_h |v_h+a_h|^p) dx + 
\beta_h\Hn(J_{v_h}\cap Q_t) +
o(1).
	\end{multline}

Recalling that (by Properties~\ref{prop:app1}. and~\ref{prop:app2}.
of Theorem~\ref{th:app})
$J_{v_h}\subset J_{\tilde u_h}\subset Q_t\setminus Q_{{t-\sqrt{\delta_h}}}$
and $\H^{n-1}(J_{\tilde u_h}\setminus J_{u_h})\le c\sqrt{\delta_h}\H^{n-1}(J_{u_h})$,
we {then subtract $\beta_h\Hn(J_{u_h}\cap Q_t\setminus Q_{t-\sqrt{\delta_h}})$ from both sides of inequality \eqref{eq:Qt} and we deduce}
\begin{multline}\label{eq:eqcutoff}
  \int_{Q_t} (f_0(e(u_h))  + \cost_h |u_h|^p) dx+\beta_h\Hn(J_{u_h}\cap Q_{t-\sqrt{\delta_h}})\\
\leq 
\int_{Q_t} (f_0(e(v_h))  + \cost_h |v_h+a_h|^p) dx + o(1).
\end{multline}

We have
$e(v_h)=\psi e(v)+(1-\psi)e(\tilde u_h)+\nabla\psi\odot (v-\tilde u_h+a_h)$.
Recalling that $\tilde u_h-a_h\to u$ in $L^p_\loc(Q_t)$  and that  $v=u$ outside of $\{\psi=1\}$, we obtain
$\tilde u_h-a_h-v\to 0$ in $L^p(\{0<\psi<1\})$.
Hence
\[
\int_{Q_t}f_0(e(v_h))dx \le (1+o(1)) \left[\int_{Q_t} \psi f_0(e(v)) dx
  + \int_{Q_t} (1-\psi)f_0(e(\tilde u_h)) dx\right] + o(1).
\]
By~\ref{prop:app-psif}.~in Theorem \ref{th:app} we find
\[
\int_{Q_t} (1-\psi)f_0(e(\tilde u_h))dx
\le \int_{Q_t}(1-\psi)f_0(e(u_h)) dx + o(1),
\]
so that subtracting $(1-\psi)f_0(e(u_h))$ from both sides of \eqref{eq:eqcutoff} we conclude
\begin{multline}\label{eq:murid}
  \int_{Q_t}  (\psi f_0( e(u_h))  + \cost_h |u_h|^p) dx
+\beta_h\Hn(J_{u_h}\cap Q_{t-\sqrt{\delta_h}})
\\\le \int_{Q_t} (\psi f_0(e(v))  + \cost_h |v_h+a_h|^p )dx 
  + o(1).
\end{multline}

Similarly, thanks to Property~\ref{prop:app-lp}.~of Theorem~\ref{th:app}
we have:
\begin{multline*}
\cost_h\int_{Q_t} (1-\psi)|\tilde u_h|^p dx 
\\
{=\cost_h\int_0^1 \Big(\int_{Q_t\cap \{1-\psi>s\}} |\tilde u_h|^p dx \Big) ds
\le
\cost_h\int_0^1 \Big(\int_{Q_t\cap \{1-\psi>s\}} |u_h|^p dx + o(1)\Big)ds}
\\
\le 
\cost_h\int_{Q_t} (1-\psi)|u_h|^p dx +o(1)
\end{multline*}
so that, since $v_h+a_h=\psi(v+a_h) + (1-\psi)\tilde u_h$, we deduce
from~\eqref{eq:murid}
\begin{multline}\label{eq:murid2}
  \int_{Q_t}  \psi (f_0( e(u_h))  + \cost_h |u_h|^p) dx
+\beta_h\Hn(J_{u_h}\cap Q_{t-\sqrt{\delta_h}})
\\\le \int_{Q_t} \psi (f_0(e(v))  + \cost_h |v+a_h|^p )dx 
  + o(1).
\end{multline}
Observe that $\cost_h^{1/p}|v+a_h|\to |\bar a|$ strongly.
Then, recalling \eqref{eq:lsc} and the convergence $\cost_h^{1/p}u_h\weakto \bar a$ in $L^p$, we obtain in the limit
\[
\int_{Q_t}\psi (f_0(e(u)) + |\bar a|^p) dx
\le \int_{Q_t}\psi (f_0(e(v))  + |\bar a|^p) dx. 
\]
Since $v=u$ outside of $\{\psi=1\}$, we may
add $(1-\psi)f_0(e(u))$ to both sides and we find
\begin{equation}\label{eq:uv}
  \int_{Q_t}f_0(e(u))dx\le \int_{Q_t}f_0(e(v)) dx
\end{equation}
which implies Property~\ref{propconvenerg-itmin}.~of the Theorem.

Next we show Property~\ref{propconvenerg-itrho}., namely that:
\begin{equation}
  \label{eq:Lambda}
  \Lambda(t)=\int_{Q_t}(f_0(e(u))+ |\bar a|^p )dx.
\end{equation}
For this we {choose} $v=u$ in~\eqref{eq:murid2}, and {let} $\bar t<t$
be such that $\psi\equiv 1$ in $Q_{\bar t}$, {then} we find as $h\to\infty$ that
\[
\Lambda(\bar t)\le \int_{Q_t}\psi (f_0(e(u))  + |\bar a|^p) dx \le \Lambda(t)
\]
thanks to~\eqref{eq:lsc2}. Since $t$ is a point of continuity of
$\Lambda$, \eqref{eq:Lambda} follows, sending $\bar t$ to $t$. In 
particular, we {deduce} that every $t\in (0,1]$ is a point of
continuity of $\Lambda$.

It remains {to show} Properties~\ref{propconvenerg-remark}.
The fact that $\lim_h \beta_h\H^{n-1}(J_{u_h}\cap Q_t)=0$
and that $e(u_h)\to e(u)$, $\cost_h^{1/p}u_h\to \bar a$ strongly
easily follows from Property~\ref{propconvenerg-itrho}.
Finally, as $t=1$ is a point of (left) continuity of $\Lambda$,
the above construction applied to $Q_1$ provides $\tilde u_h$ such
that $(\tilde u_h-a_h)_h$ converges
strongly to $u$ in $L^p_\loc(Q_1)$, and therefore, also
$u_h\chi_{Q_1\setminus \tilde \omega_h}$ converges to $u$ in $L^p_\loc(Q_1)$.
Hence, possibly extracting a last subsequence, we may also assume
that $u_h\to u$~a.e.~in $Q_1$.
\end{proof}

{We conclude this section by stating explicitly the compactness and semicontinuity result for functions with vanishing jump sets which has been proved within the proof of Theorem \ref{t:cm}.}
\begin{corollary}[Compactness and semicontinuity]\label{t:complsc}
	Let $p\in(1,\infty)$ and let $u_k\in GSBD^p(Q_1)$ be such that	
	 \begin{equation}\label{e:uhhp}
	 \sup_k \int_{Q_1} \fz(e(u_k))dx<\infty, \qquad {\delta_k}:=
{\Hn(J_{u_k})^{1/n}}\to0.
	 \end{equation}
Then there are a subsequence $k_h$ of $k$, a sequence $\tilde u_h\in GSBD^p(Q_1){\cap C^\infty(Q_{{1-\sqrt{\delta_{k_h}}}})}$ with $\{\tilde u_h\neq u_{k_h}\}\subset\subset Q_1$, affine functions $a_h:\R^n\to\R^n$ with $e(a_h)=0$, and $u\in W^{1,p}(Q_1{;\R^n})$, such that 
\begin{enumerate}
	\item $\tilde u_h-a_h\to u$ {in} $L^p_{\mathrm{loc}}(Q_1)$ and
	$u_{k_h}-a_h\to u$ $\calL^n$-a.e. on $Q_1$;
	\item for every $r,s$ with $r<s<1$ we have 
	$$\int_{Q_1\setminus Q_s} f_0(e(\tilde u_h))dx\leq (1+o(1))\int_{Q_1\setminus Q_{r}} f_0(e(u_{k_h}))dx,$$
	$$\Hn(J_{\tilde u_h}\cap (Q_1\setminus Q_r))\leq (1+o(1))\Hn(J_{u_{k_h}}\cap (Q_1\setminus Q_r)),$$
{	$$\int_{Q_1\setminus Q_r} |\tilde u_h|^pdx\leq \int_{Q_1\setminus Q_r} |u_{k_h}|^pdx + o(1)\int_{Q_1}\big(|u_{k_h}|^p+f_0(e(u_{k_h}))\big)dx;$$}
	\item the following lower estimate holds for all ${t}\in (0,1]$:
	\begin{equation}\label{e:tesiuh}
	\int_{Q_{t}} \fz(e(u))dx \le {\liminf_{k\to\infty} \int_{Q_{t}} \fz(e(u_{k}))dx\,.}
	\end{equation}	
\end{enumerate}
\end{corollary}
\begin{proof}
For large $h$ we can apply Theorem \ref{th:app} to obtain  $\tilde u_k\in GSBD(Q_1)$  from $u_k$.
By Properties \ref{prop:app1}. and \ref{prop:app3}.
	 up to a subsequence and a rigid motion
	$a_k$ the sequence $\tilde u_k$ converges in $L^p_\loc(Q_1)$
	to a limit $u\in W^{1,p}(Q_1)$. Moreover $(u_k-a_k)\chi_{Q_1\setminus \tilde \omega_k}$
	converges also to $u$ in $L^p_\loc(Q_1)$ by Theorem \ref{th:app}, Property~\ref{prop:app-u}. Since $|\tilde \omega_k|\to0$
	we can deduce that $u_k-a_k$ converges to $u$ in measure and then up to subsequences also almost everywhere in $Q_1$.

	The assertion in 2.~directly follows by Theorem \ref{th:app}, Properties \ref{prop:app2}., \ref{prop:app-u}., \ref{prop:app-psif}., {and \ref{prop:app-lp}}.
Let now $0<{t}'<{t}<1$.  Then $e(\tilde u_k)\weakto e(u)$ in $L^p(Q_{t})$, 
	and using convexity of $f_0$ and then
  Property~\ref{prop:app3}.~of Theorem \ref{th:app} we obtain
	\[
	\int_{Q_{{t}'}} f_0(e(u)) dx\le \liminf_k \int_{Q_{{t}'}} f_0(e(\tilde u_k)) dx
	\le \liminf_k \int_{Q_{t}} f_0(e(u_k)) dx.
	\]
	By continuity of the left-hand side we deduce first the case ${t}'={t}$ and then the case ${t}=1$.
\end{proof}

\section{{Existence of Griffith's minimizers in dimension \boldmath{$n$}}}\label{s:exis}
The following statements follow from Theorem \ref{t:cm} with exactly the same proof as \cite[{Lemma 3.8 and Corollary 3.9}]{ContiFocardiIurlano_exist} respectively.

{The following density lower bound is explicitly stated for $p=2$.
In that case, it is well-known that solutions of $\Div\C e(u)=0$ are
smooth and can be computed using a ($\C$-dependent) kernel
which is $(2-{n})$-essentially homogeneous (see for
instance~\cite[\S~6.2, Thm.~6.2.1]{Morrey-Multiple}). It follows that (for a
solution defined in the unit ball) 
\[
\|e(u)\|_{L^\infty(B_{1/2})} \le c \|e(u)\|_{L^2(B_1)}
\]
for some $c$ depending only on $\C$ and thus, for $\rho\le 1/2$,
\[
\int_{B_\rho} f_0(e(u))dx \le c\rho^n \int_{B_1} f_0(e(u)) dx.
\]
Therefore, one can reproduce the proof
of~\cite[Lemma~3.7]{ContiFocardiIurlano_exist}
with almost no change, and obtain the following result.
}

\begin{lemma}[Density lower bound for $G_0$]\label{l:dens}
	Let $p=2$,
        $\cost\ge 0$, and $\beta>0$. 
	If $u\in GSBD^2(\Omega)$ is a local minimizer of
	$G(\cdot,\cost,\beta,\Omega)$ defined in \eqref{e:G}, then there exist $\vartheta_0$ 
	and $R_0$, depending only on ${n}$, $\mathbb{C}$, $\cost$, $\beta$, $\parametermu$, and 
	$\|g\|_{L^\infty(\Omega)}$, such that if $0<\rho<R_0$, $x\in\overline{J_u}$, and 
	$B_\rho(x)\subset\hskip-0.125cm\subset\Omega$, then 
	\be{\label{e:lb}
		G_0(u,\cost,\beta,B_\rho(x))\geq \vartheta_0\rho^{n-1}.}
	Moreover
	\be{\label{e:saltochiuso}
		\calH^{n-1}(\Omega\cap(\overline{J_u}\setminus J_u))=0.
	}
\end{lemma}
\begin{corollary}[{Density lower bound for the jump}]
	Let $p=2$, $\cost\ge 0$, and $\beta>0$. 
	If $u\in GSBD^2(\Omega)$ is a local minimizer of
	$G(\cdot,\cost,\beta,\Omega)$ defined in \eqref{e:G}, then there exist 
	$\vartheta_1$ and $R_1$, depending only on {$n$}, $\mathbb{C}$, $\cost$, $\beta$, 
	$\parametermu$, and $\|g\|_{L^\infty(\Omega)}$, 
	such that if $0<\rho<R_1$, 
	$x\in\overline{J_u}$, and $B_\rho(x)\subset\hskip-0.125cm\subset\Omega$, then 
	\be{\label{e:dlbjump}
		\calH^{n-1}(J_u\cap B_\rho(x))\geq \vartheta_1\rho^{n-1}.}
\end{corollary}
{Theorem \ref{t:mainreg}} {follows from Lemma \ref{l:dens}, see \cite[Theorem 1.2]{ContiFocardiIurlano_exist} for the {details of the} proof.}
{Theorem \ref{t:main} follows from Theorem \ref{t:mainreg}  and the existence of weak global solutions proven in 
 \cite[Theorem 11.3]{gbd}.}

\section*{Acknowledgments} 
F. Iurlano has been supported by the program FSMP and the project Emergence Sorbonne-Université ANIS.
{S. Conti has been partially supported by the Deutsche Forschungsgemeinschaft through the Sonderforschungsbereich 1060 {\sl ``The mathematics of emergent effects''}.}
{The authors wish to thank G. Friesecke for useful references.}

\bibliographystyle{siam}
\bibliography{biblio}

\end{document}